\documentclass[11pt]{article}
\usepackage{amsmath,amsfonts,amssymb,amscd}

\newtheorem{theo}{Theorem}
\newtheorem{lem}{Lemma}[section]
\newtheorem{prop}{Proposition}[section]

\newtheorem{cor}{Corollary}[section]
\newtheorem{rem}{Remark}[section]
\newtheorem{dfn}{Definition}[section]

\makeatletter \@addtoreset{equation}{section} \makeatother

\newcommand{\mR}{\mathbb{R}}

\newcommand{\mZ}{\mathbb{Z}}
\newcommand{\mN}{\mathbb{N}}

\newcommand{\calH}{{\cal H}}

\newcommand{\calM}{{\cal M}}

\newcommand{\thet}{\vartheta}

\newcommand{\Conj}{\operatorname{Conj}}

\newcommand\qed{{\unskip\nobreak\hfil\penalty50
  \hskip2em\hbox{}\nobreak\hfil\mbox{\rule{1ex}{1ex} \qquad}
    \parfillskip=0pt \finalhyphendemerits=0\par\medskip}}

\begin{document}

\title
{Integrable perturbations of polynomial Hamiltonian systems}
\author{ D.~Treschev\\
Steklov Mathematical Institute of Russian Academy of Sciences
\footnote{
This work was supported by the Russian Science Foundation under grant no. 25-11-00114.
}
}

\maketitle

\begin{abstract}
We consider a Hamiltonian system on the symplectic space $(\mR^{2n}, dy\wedge dx)$ with a real-analytic Hamiltonian $H : \mR^{2n}\to \mR$. We assume that the system has a non-degenerate equilibrium position at the origin. Under some nonresonance assumptions we prove the following.

For any positive integer $M$ there exists a real-analytic function $F:\mR^{2n}\to\mR$ such that

(1) $F = O\big( (|x|+|y|)^{M+1} \big)$ at the origin,

(2) the system with Hamiltonian $H+F$ is completely integrable in $\mR^{2n}$.
\end{abstract}

\section{Local result}

Let $x = (x_1,\ldots,x_n)$ and $y = (y_1,\ldots,y_n)$ be coordinates in the symplectic space
$(\mR^{2n},dy\wedge dx)$. Let $H:\mR^{2n}\to\mR$ be a real-analytic function. We assume that the Hamiltonian system
\begin{equation}
\label{eq1.1}
  \dot x = \partial H/\partial y, \quad
  \dot y = - \partial H/\partial x
\end{equation}
has a non-degenerate equilibrium position at the origin and the corresponding eigenvalues are pairwise distinct. This means (see for example, \cite{Arn_MMCM}) that after a linear change of coordinates $H$ takes the form\footnote{Here we use the notation $O_M(x,y)$ as a shorter form of $O\big( (|x|+|y|)^M \big)$.}
\begin{eqnarray}
\nonumber
&&  H = H_2 + H_*, \qquad
    H_* = O_3(x,y), \\
\nonumber
     H_2
 &=& \sum_{j=1}^{2n_1} \Big( - a_j (y_{2j-1} x_{2j-1} + y_{2j} x_{2j})
                             + b_j (y_{2j-1} x_{2j} - y_{2j} x_{2j-1}) \Big) \\
%\label{H2}
\nonumber
 &&\quad  + \sum_{k=2n_1+1}^{n_2} \frac{\omega_k}2 (x_k^2 + y_k^2) \;
          + \sum_{l=n_2+1}^{n} \lambda_l x_l y_l.
\end{eqnarray}
The eigenvalues $\mu_1,\ldots,\mu_n$ are
\begin{equation}
\label{eq1.2}
\begin{array}{ccc}
& \mu_{2j-1} = - a_j - ib_j, \quad
    \mu_{2j} = - a_j + ib_j, \qquad
  j = 1,\ldots,n_1, & \\[1mm]
& \mu_k = -i\omega_k, \quad \mu_l = \lambda_l, \qquad
  k = 2n_1+1,\ldots,n_2, \quad
  l = n_2+1,\ldots,n. &
\end{array}
\end{equation}

The vector of eigenvalues $\mu = (\mu_1,\ldots,\mu_n)$ is said to be nonresonant if
$$
  \langle \mu,k\rangle \ne 0 \quad
  \mbox{ for any $k\in\mZ^n\setminus\{0\}$}.
$$

We start with a very simple result.

\begin{theo}
\label{theo1}
Suppose that $H$ is a polynomial, $\deg H \le M$ and $\mu$ is nonresonant. Then there exists a function $F$ such that
\begin{itemize}
\item  $F$ is real-analytic in a neighborhood $U$ of the point $0\in\mR^{2n}$,
\item  $F = O_{N+1}(x,y)$ at the origin,
\item  system with Hamiltonian $H_2 + H_* + F$ is completely integrable in $U$.
\end{itemize}
\end{theo}

{\it Proof}. Let $\Phi$ be a real-analytic symplectic change of coordinates
$$
  (x,y)\mapsto \Phi(x,y) = (X,Y), \qquad
  \Phi(0) = 0, \quad D\Phi(0) = I
$$
which reduces the Hamiltonian to the normal form $N(X,Y) + O_{M+1}(X,Y)$
up to terms of orders higher than $M$. By the nonresonance assumption the function $N$ is a polynomial in
\begin{equation}
\label{eq1.3}
  Y_{2j-1} X_{2j-1} + Y_{2j} X_{2j}, \quad
  Y_{2j-1} X_{2j} - Y_{2j} X_{2j-1}, \quad
  X_k^2 + Y_k^2, \quad
  X_l Y_l.
\end{equation}

The system with Hamiltonian $N$ is completely integrable: the quadratic polynomials (\ref{eq1.3}) are first integrals in involution. The system with Hamiltonian $N\circ\Phi^{-1}$ is also completely integrable. Moreover,
$$
  F := N\circ\Phi^{-1} - (H_2 + H_*) = O_{M+1}(x,y).
$$
\qed

\begin{rem}
\label{rem1}
(a) We may assume in Theorem \ref{theo1} that $H_* = O_3(x,y)$ is any function real-analytic at the origin.
\smallskip

(b) Nonresonance condition for $\mu$ may be replaced by the condition of complete integrability of the partial normal form $N$ (in the resonant case $N$ is not a function of the polynomials (\ref{eq1.3}) only.

(c) The case $n_1=0$ and $n_2=n$ is called elliptic. It is more interesting from dynamical viewpoint because if $n_1\ne 0$ or $n_2\ne n$, almost all trajectories of the system (\ref{eq1.1}) starting in a small neighborhood of the origin, leave this neighborhood for both large positive and large negative values of time.
\end{rem}

\section{Extension of the neighborhood $U$}

Radius of convergence $r$ of the transformation $\Phi$ (as well as of the series $N\circ\Phi$) is positive. However we do not expect any constructive lower estimate for $r$. Here is a version of Theorem \ref{theo1} with a ``big'' neighborhood $U$.

\begin{theo}
\label{theo2}
Theorem 1 holds with $U = \mR^{2n}$.
\end{theo}

{\it Proof}. We follow essentially the same idea. We construct the above coordinate change $\Phi$ as a shift along solutions of another Hamiltonian system
\begin{equation}
\label{eq2.1}
  x' = \partial K / \partial y, \quad
  y' = - \partial K / \partial x, \qquad
  (\cdot)' = d / d\delta, \quad
  0\le \delta\le +\infty
\end{equation}
with the Hamiltonian $K = K(x,y,\delta)$. The shift has the form
\begin{equation}
\label{eq2.2}
  (x,y) = (x(0),y(0)) \mapsto (X,Y) = (x(+\infty),y(+\infty)).
\end{equation}

\begin{lem}
\label{lem2.1}
There exists a Hamiltonian $K$ such that (\ref{eq2.2}) transforms $H$ to the partial normal form $N+O_{M+1}(x,y)$, where $N$ is a polynomial function of the quadratic polynomials (\ref{eq1.3}). Moreover,
\begin{itemize}
\item $K$ is a polynomial in $x$ and $y$, $\deg K = M$,
\item coefficients of this polynomial tend to zero exponentially in $\delta$.
\end{itemize}
\end{lem}

Proof of the lemma requires a special technics, which is called the {\it continuous averaging}. We prove Lemma \ref{lem2.1} in Section \ref{sec:ca}.
\medskip

\begin{lem}
\label{lem2}
Let $K$ be a function from Lemma \ref{lem2.1}. There exists a function $P(x,y,\delta)$ such that
\begin{itemize}
\item $P$ is a polynomial in $x$ and $y$, $\deg P \le M$,
\item $P(x,y,\delta) e^{-(x^2 + y^2)} = K(x,y,\delta) + O_{M+1}(x,y)$,
\item coefficients of the polynomial $P$ tend to zero exponentially as $\delta\to +\infty$.
\end{itemize}
\end{lem}

{\it Proof}. The polynomial $P$ can be computed explicitly:
$$
  P = \mbox{polynomial part (degree $\le M$) in the Taylor expansion in $x,y$ of $K e^{x^2 + y^2}$}.
$$
\qed

We take $L(x,y,\delta) = P(x,y,\delta) e^{-(x^2 + y^2)}$. Then the function $L$
\begin{itemize}
\item is real-analytic on $\mR^{2n}$,
\item satisfies the equation $L(x,y,\delta) - K(x,y,\delta) = O_{M+1}(x,y)$,
\item tends to zero exponentially when $|x| + |y| \to \infty$:
$$
  \lim_{|x|+|y|\to\infty} L(x,y,\delta) e^{(x^2+y^2)/2} = 0,
$$
\item tends to zero uniformly exponentially on $\mR^{2n}$ when $\delta\to +\infty$.
$$
  \lim_{\delta\to +\infty} L(x,y,\delta) e^{\alpha\delta} = 0 \quad
  \mbox{for some } \alpha > 0.
$$
\end{itemize}

Solutions of the system with Hamiltonian $L$ are globally defined on the infinite interval
$\delta\in [0,+\infty)$. Each solution $(x(\delta),y(\delta))$ exponentially tends to a point $(X,Y)$ (depending on the solution):
$$
    \lim_{\delta\to +\infty} (x(\delta),y(\delta))
  = (X,Y)\in\mR^{2n}.
$$
Hence, the map
$$
  (x(0),y(0)) \mapsto \Psi(x(0),y(0)) = (X,Y)
$$
is a global real-analytic diffeomorphism of $(\mR^{2n},0)$ to itself. The Hamiltonian $N\circ\Psi^{-1}$
\begin{itemize}
\item is real-analytic on $\mR^{2n}$,
\item generates a completely integrable Hamiltonian system,
\item $N\circ\Psi^{-1} - (H_2 + H_*) = O_{N+1}(x,y)$\quad at $0\in\mR^{2n}$.
\end{itemize}
Theorem \ref{theo2} is proved. \qed

\begin{rem}
\label{rem2}
Having in mind the Grauert theorem on an immersion of a real-analytic manifold in $\mR^D$, $D\in\mN$, \cite{Gra}, it is possible to replace in Theorem \ref{theo2} the phase space $\mR^{2n}$ and the Hamiltonian $H$ by any compact real-analytic symplectic manifold $\calM$ and a real-analytic function on $\calM$ with a nondegenerate
equilibrium position.
\end{rem}

\section{Proof of Lemma \ref{lem2.1}}
\label{sec:ca}

The proof consists of several steps.

\subsection{``Diagonalization'' of $H_2$}

It is convenient to use complex coordinates
$$
   (z,w) = (z_1,\ldots,z_n,w_1,\ldots,w_n), \quad (x,y) =  \thet(z,w), \qquad
   dy\wedge dx = dw\wedge dz,
$$
where for $j=1,\ldots,n_1$
\begin{eqnarray*}
&\displaystyle
    z_{2j-1}
  = \frac{x_{2j-1} + x_{2j}}2 - i\frac{x_{2j-1} - x_{2j}}2, \quad
    z_{2j}
  = \frac{x_{2j-1} + x_{2j}}2 + i\frac{x_{2j-1} - x_{2j}}2, &\\
&\displaystyle
    w_{2j-1}
  = \frac{y_{2j-1} + y_{2j}}2 + i\frac{y_{2j-1} - y_{2j}}2, \quad
    w_{2j}
  = \frac{y_{2j-1} + y_{2j}}2 - i\frac{y_{2j-1} - y_{2j}}2, &
\end{eqnarray*}
for $k=2n_1+1,\ldots,n_2$
$$
  z_k = \frac1{\sqrt2}(iy_k + x_k), \quad
  w_k = \frac1{\sqrt2}(y_k + ix_k),
$$
while for $l=n_2+1,\ldots,n$ we have $z_l = x_l, w_l = y_l$.

The Hamiltonian $H$ takes the form
$$
  \hat H = H\circ\thet = \hat H_2 + \hat H_*, \qquad
  \hat H_2 = H_2\circ\thet = \sum_{j=1}^n \mu_j z_j w_j, \quad
  \hat H_* = H_*\circ\thet,
$$
where the coefficients $\mu_j$ are computed from (\ref{eq1.2}).

\subsection{Reality condition}

Let $\Conj$ be the operator such that for any
$G = \sum_{|\alpha|+|\beta|\ge 2} G_{\alpha,\beta} x^\alpha y^\beta$
$$
    \Conj(G)(x,y) = \overline{G(\overline x,\overline y)}, \quad
    \mbox{or equivalently}, \quad
    \Conj(G)(x,y)
  = \sum_{|\alpha|+|\beta|\ge 2} \overline{G_{\alpha,\beta}} x^\alpha y^\beta.
$$

The original Hamiltonian $H$ is real: $H = \Conj(H)$. This implies that $\hat H$ should also satisfy a certain reality condition. Consider the operators $\Theta$ and $\Conj_\thet$, defined by the equations
$$
  \Theta (H) = H\circ\thet, \quad
  \Conj_\thet = \Theta\Conj\Theta^{-1}.
$$
Then the reality condition for $\hat H$ is given by the following definition:

\begin{dfn}
The function $\hat H = \hat H(z,w)$ is said to be $\thet$-real iff $\Conj_\thet \hat H = \hat H$.
\end{dfn}

Obviously $H$ is real if and only if $\hat H = \Theta H$ is $\thet$-real.

The following lemma gives another version of $\thet$-reality.

\begin{lem}
\label{lem3.1}
The series
\begin{equation}
\label{eq3.1}
  \hat G = \sum_{|\alpha|+|\beta|\ge 2} \hat G_{\alpha,\beta} z^\alpha w^\beta
\end{equation}
is $\thet$-real if $\thet$-real iff for any multiindices $\alpha,\beta$
\begin{equation}
\label{eq3.2}
    \hat G_{\alpha,\beta}
  = i^{-\alpha_{2n_1+1}-\ldots-\alpha_{n_2} - \beta_{2n_1+1}-\ldots-\beta_{n_2}}
    \overline{\hat G_{\alpha',\beta'}},
\end{equation}
where
\begin{eqnarray*}
     \alpha'_{2j-1} = \alpha_{2j}, \quad
     \alpha'_{2j} = \alpha_{2j-1}, \quad
     \beta'_{2j-1} = \beta_{2j}, \quad
     \beta'_{2j} = \beta_{2j-1} \quad
  &\mbox{for any}& j = 1,\ldots,n_1, \\
     \alpha'_k = \beta_k, \quad
     \beta'_k = \alpha_k \quad
  &\mbox{for any}& k = 2n_1+1,\ldots,n_2, \\
     \alpha'_l = \alpha_l, \quad
     \beta'_l = \beta'_l \quad
  &\mbox{for any}& l = n_2+1,\ldots,n.
\end{eqnarray*}
\end{lem}

{\it Proof}. It is sufficient to check equation (\ref{eq3.2}) for
$\hat G = c z^\alpha w^\beta + c' z^{\alpha'} w^{\beta'}$. This can be done by direct computation. \qed

\begin{prop}
If $\hat G_1$ and $\hat G_2$ are $\thet$-real and $c\in\mR$ then $c\hat G_1$, $\hat G_1+\hat G_2$, $\hat G_1\hat G_2$, and $\{\hat G_1,\hat G_2\}$ are also real.
\end{prop}

We skip an obvious proof. \qed

\subsection{Continuous averaging}

We construct the polynomial $K$ by using the method of continuous averaging \cite{TZ}
In the context of the theory of normal forms this method is presented in \cite{Tre1,Tre2}. The main idea is to look for $K(x,y,\delta)$ in the form $\xi\calH_*$, where $H_2 + \calH_*$
is the Hamiltonian $H$ in the variables $x(\delta),y(\delta)$, obtained as the $\delta$-shift along solutions of (\ref{eq2.1}):
\begin{equation}
\label{eq3.3}
  H_2(x(\delta),y(\delta)) + \calH_*(x(\delta),y(\delta),\delta) = H_2(x,y) + H_*(x,y)
\end{equation}
The linear operator $\xi$ will be specified below.

Differentiating (\ref{eq3.3}) in $\delta$, we obtain the initial value problem
\begin{equation}
\label{eq3.4}
  \partial_\delta\calH_* = - \{\xi\calH_*,H_2 + \calH_*\}, \qquad
  \calH_*\big|_{\delta=0} = H_*,
\end{equation}
where $\{\,,\}$ is the Poisson bracket.

We put $\hat\calH = \Theta\calH$ and determine the operator $\hat\xi$ by the identity
$$
  \xi G = \hat\xi\hat G \quad
  \mbox{for any function } \quad  G = \Theta^{-1}\hat G = O_3(x,y).
$$
Then (\ref{eq3.4}) takes the form
\begin{equation}
\label{eq3.5}
  \partial_\delta\hat\calH_* = - \{\hat\xi\hat\calH_*,\hat H_2 + \hat\calH_*\}, \qquad
  \hat\calH_*\big|_{\delta=0} = \hat H_*.
\end{equation}
For any $\hat G$, satisfying (\ref{eq3.1}), we put
\begin{equation}
\label{eq3.6}
    \hat\xi\hat G
  = \sum_{|\alpha|+|\beta|\le M,\,\langle\mu,\beta-\alpha\rangle\ne 0}
                \sigma_{\alpha,\beta} \hat G_{\alpha,\beta} z^\alpha w^\beta, \qquad
    \sigma_{\alpha,\beta}
  = \frac{|\langle\mu,\beta-\alpha\rangle|}{\langle\mu,\beta-\alpha\rangle}.
\end{equation}

\begin{prop}
\label{prop3.2}
For any real $\hat G$ the function $\hat\xi\hat G$ is also $\thet$-real.
\end{prop}

{\it Proof}. By (\ref{eq1.2}) we have:
$\langle \mu,\beta'-\alpha' \rangle = \overline{\langle \mu,\beta-\alpha \rangle}$. It remains to use Lemma
\ref{lem3.1}. \qed

\begin{cor}
\label{cor3.1}
If the initial condition $\hat H$ in (\ref{eq3.5}) is real then the solution $\hat\calH$ is also real.
\end{cor}

Now Lemma \ref{lem2.1} follows from Corollary \ref{cor3.1} and Proposition \ref{prop3.3} (below) on the existence of a solution of (\ref{eq3.5}). The function $K$ from Lemma \ref{lem2.1} is computed from the equation $K = \Theta^{-1}(\hat K)$, where $\hat K$ is determined by Assertion (2) of Proposition \ref{prop3.3}.

\begin{prop}
\label{prop3.3}
Suppose $\hat H_* = O_3(z,w)$ is a power series in $z$ and $w$. Then there exists a unique formal solution $\hat\calH$ of the system (\ref{eq3.5})--(\ref{eq3.6}) on the interval $\delta\in [0,+\infty)$ in the form of a power series in $z$ and $w$. Moreover,

(1) the polynomial part of degree $M$ in the Taylor expansion of $\hat\calH$ at the origin exponentially tends to the normal form as $\delta\to+\infty$,

(2) the function $\hat K = \xi\hat\calH$ is a polynomial of degree (at most) $M$ in $z$ and $w$ with coefficients exponentially tending to zero as $\delta\to +\infty$.
\end{prop}

\begin{rem}
For any $\delta\ge 0$ the function $\hat\calH$ is analytic in $z$ and $w$ in a small (but independent of $\delta$) neighborhood of the origin. This may be proven with the help of the majorant method (proofs of analogous statements may be found in \cite{Tre1,Tre2}).
\end{rem}

\subsection{Proof of Proposition \ref{prop3.3}}

If in (\ref{eq3.5}) $\hat\calH_* = \sum_{|\alpha|+|\beta|>2} \hat\calH_{\alpha,\beta} z^\alpha w^\beta$, we have by (\ref{eq3.6}):
$$
    - \{\hat\xi\hat\calH_*,\hat\calH_2\}
  = - \sum_{2<|\alpha|+|\beta|\le M}
           |\langle\mu,\beta-\alpha\rangle| \hat\calH_{\alpha,\beta} z^\alpha w^\beta.
$$
Hence (\ref{eq3.5}) takes the form
\begin{eqnarray}
\label{eq3.7}
      \partial_\delta\hat\calH_{\alpha,\beta}
  &=& - \rho_{\alpha,\beta} |\langle\mu,\beta-\alpha\rangle| \hat\calH_{\alpha,\beta}
      - \{\hat\xi\hat\calH_*,\hat\calH_*\}_{\alpha,\beta}, \qquad
      \hat\calH_{\alpha,\beta}(0) = \hat H_{\alpha,\beta}, \\
\nonumber
      \rho_{\alpha,\beta}
  &=& \left\{\begin{array}{cl} 1 & \mbox{ if } 2 < |\alpha| + |\beta| \le M, \\
                               0 & \mbox{ if } |\alpha| + |\beta| > M.
             \end{array}
      \right.
\end{eqnarray}
The term $\{\hat\xi\hat\calH_*,\hat\calH_*\}_{\alpha,\beta}$ denotes the coefficient at $z^\alpha w^\beta$ in the Taylor expansion of $\{\hat\xi\hat\calH_*,\hat\calH_*\}$.

For any $\alpha$ and $\beta$ the function $\{\hat\xi\hat\calH_*,\hat\calH_*\}_{\alpha,\beta}$ is a quadratic polynomial in the coefficients $\hat\calH_{\alpha',\beta'} = \hat\calH_{\alpha',\beta'}(\delta)$, where
\begin{equation}
\label{eq3.8}
  2 < |\alpha'| + |\beta'| < |\alpha| + |\beta|.
\end{equation}
More precisely,
$$
      \{\hat\xi\hat\calH_*,\hat\calH_*\}_{\alpha,\beta}
  =   \sum_{2 < |\alpha'| + |\beta'| \le M} \sum_{j=1}^n
      \rho_{\alpha',\beta'} \Big( \beta'_j (\alpha_j + 1) - \alpha'_j (\beta_j + 1) \Big)
      \hat\calH_{\alpha',\beta'} \hat\calH_{\alpha+e_j-\alpha',\beta+e_j-\beta'},
$$
where the vector $e_j=(e_{j1},\ldots,e_{jn})$ is such that $e_{jk} = \delta_{jk}$.

In this sense the system (\ref{eq3.7}) has a triangular form and may be solved by induction. Indeed, if
$|\alpha| + |\beta| = 3$, we have the equations
$$
      \partial_\delta\hat\calH_{\alpha,\beta}
   =  - |\langle\mu,\beta-\alpha\rangle| \hat H_{\alpha,\beta}.
$$
Hence, $\hat\calH_{\alpha,\beta}(\delta) = \hat H_{\alpha,\beta} e^{-|\langle\mu,\beta-\alpha\rangle|\delta}$.

Then by using induction arguments, we prove that for any $\alpha,\beta$ such that $2<|\alpha|+|\beta|\le M$ we have:
$$
    \hat\calH_{\alpha,\beta}(\delta)
  = \Big(\hat H_{\alpha,\beta} + P_{\alpha,\beta}(\hat H_*,\delta)\Big)
                                 e^{-|\langle\mu,\beta-\alpha\rangle|\delta},
$$
where $P_{\alpha,\beta}$ is a polynomial in $\hat H_{\alpha',\beta'}$ with $\alpha',\beta'$, satisfying (\ref{eq3.8}) and coefficients in the form of (finite) linear combinations of terms $\delta^s e^{-\nu\delta}$, $s\in\mZ_+$. Here $\nu\ge 0$ and moreover, if the term $\delta^s e^{-\nu\delta}$ with $\nu=0$ appears in a coefficient of $P_{\alpha,\beta}$ with $\langle\mu,\beta-\alpha\rangle = 0$ then in this term $s=0$.

\end{document}